\documentclass[11pt,a4paper]{article}

\usepackage{amsmath,amssymb,amsthm}
\usepackage{amsfonts}
\usepackage{graphicx}
\usepackage{hyperref}
\usepackage{enumerate}
\usepackage{mathtools}

\newtheorem{theorem}{Theorem}[section]
\newtheorem{lemma}[theorem]{Lemma}

\newtheorem{corollary}[theorem]{Corollary}

\newtheorem{claim}{Claim}

\theoremstyle{definition}

\theoremstyle{remark}

\newcommand{\ex}{\operatorname{ex}}

\newcommand{\F}{\mathbb{F}}
\newcommand{\PG}{\operatorname{PG}}
\newcommand{\spanop}{\operatorname{span}}

\begin{document}

\title{On the generalized Tur\'{a}n number of the complete bipartite graph $K_{3,b+1}$ \thanks{This research was supported by National Key Research and Development Program of China (No.~2023YFA1010203), National Natural Science Foundation of China (Nos. 12401464, 12471334, 12271337, and 12371347).}
}

\date{}

\author{Jing Wang$^{a}$,   Zixuan Yang$^{b}$\thanks{Corresponding author.}, Junpeng Zhou$^{c,d}$,
~\\[3mm]
\small $^{a}$School of Mathematics and Statistics,  Shaanxi Normal University, Xi'an, Shaanxi, P.R. China\\
\small $^{b}$School of Mathematics,  Northwest University, Xi'an, Shaanxi, P.R. China\\
\small $^{c}$ Department of Mathematics, Shanghai University, Shanghai, P.R. China.\\
\small$^{d}$ Newtouch Center for Mathematics of Shanghai University, Shanghai, P.R. China
}
\maketitle

\begin{abstract}
For graphs $F$ and $H$, let $\mathrm{ex}(n,H,F)$ denote the maximum number of copies of $H$ in an $n$-vertex $F$-free graph. Very recently,  Janzer, Longbrake, and Yepremyan proved that for $3<a\leq b$ and sufficiently large $t$,
\begin{equation*}
\mathrm{ex}(n,K_{a,b},K_{3,t})=\Theta_{a,b,t}(n^3).
\end{equation*}
Later, Hou, Hu, and Wang made this threshold explicit by showing that the conclusion holds for all $t\geq 2\max\{3,\lceil b/2\rceil\}+1$. In particular, for every even $b\geq 6$, this matches the necessary threshold $t=b+1$. 
In this paper, we resolve the remaining case where $b$ is odd. More precisely, we prove that for all fixed integers $b\geq 5$ and $3<a\leq b$,
\begin{equation*}
\mathrm{ex}(n,K_{a,b},K_{3,b+1})=\Theta_{a,b}(n^3).
\end{equation*}

Our construction uses a finite-field point set in $\mathrm{PG}(5,q)$ together with an orthogonal polarity. The key new ingredient is the polynomial splitting lemma due to Andrade, Bary-Soroker, and Rudnick, which produces many planes whose intersections with the point set and their polar planes both have size $b$. This gives a $K_{3,b+1}$-free
incidence graph while preserving $\Omega_{a,b}(n^3)$ copies of $K_{a,b}$.
\end{abstract}


{\noindent{\bf AMS (2020) subject classifications:} 05C35, 05C50}

%%%%%%%%%%%%%%%%%%%%%%%%%%%%%%%%%%%%%%

\setcounter{footnote}{0}
\renewcommand{\thefootnote}{}
\footnotetext{E-mail addresses:  {\tt jingwang\_math@snnu.edu.cn (J. Wang), yangzixuan@nwpu.edu.cn (Z. Yang),  junpengzhou@shu.edu.cn.(J. Zhou)}}

%%%%%%%%%%%%%%%%%%%%%%%%%%%%%%%%%%%%%%

\section{Introduction}
Given a graph $F$, a graph $G$ is \textit{$F$-free} if $G$ does not contain $F$ as a subgraph. The \textit{Tur\'{a}n number} of $F$, denoted by ${\rm{ex}}(n,F)$, is the maximum number of edges in an $n$-vertex $F$-free graph. 

The Tur\'{a}n problem is one of the central topics in Extremal Graph Theory. 
A classical result is the Tur\'{a}n theorem \cite{Tu}, which determines the exact Tur\'{a}n number of the complete graph. %$K_\ell$ on $\ell$ vertices. 
The Erd\H{o}s--Stone--Simonovits theorem \cite{ESi,ESt} gives an asymptotics of the Tur\'{a}n number for any $k$-chromatic graph with $k\geq3$. 
When $F$ is bipartite, the problem of determining $\ex(n,F)$ remains an active topic in extremal graph theory. For an extensive overview of the historical development, we refer the reader to the monograph by Bollob\'{a}s \cite{Bo}. In particular, the K\H{o}v\'{a}ri--S\'{o}s--Tur\'{a}n theorem \cite{Ko-So-Tu} establishes the upper bound $\ex(n, K_{s,t}) = O(n^{2-1/s})$, where $K_{s,t}$ denotes the complete bipartite graph with $s\le t$. 

Given two graphs $H$ and $F$, Alon and Shikhelman \cite{alon2016many} initiated the systematic study of the problem of maximizing the number of copies of $H$ in an $n$-vertex $F$-free graph. This problem, often referred to as the \textit{generalized Tur\'{a}n problem}, has attracted considerable attention in recent years. For a very recent survey on generalized Tur\'{a}n problems, one can refer to the work of Gerbner and Palmer \cite{GePa}. 

We investigate the generalized Tur\'{a}n problem for complete bipartite graphs. For integers $1 \le a \le b$ and $1 \le s \le t$, let $\mathrm{ex}(n, K_{a,b}, K_{s,t})$ denote the maximum number of copies of $K_{a,b}$ that can appear in an $n$-vertex $K_{s,t}$-free graph. The asymptotic order of magnitude of $\mathrm{ex}(n, K_{a,b}, K_{s,t})$ depends  on  $a$ and $s$.
The case $a \le s$ is relatively well understood.
When $a = b = 1$,  the problem reduces to the classical Zarankiewicz problem.
The \ K\H{o}v\'{a}ri--S\'{o}s--Tur\'{a}n theorem \cite{Ko-So-Tu} gives the upper bound $\mathrm{ex}(n, K_{s,t}) = O(n^{2-1/s})$, and matching lower bounds are known to hold when $t$ is sufficiently large compared to $s$; see \cite{Kol, alon1999, Buk}.

Alon and Shikhelman \cite{alon2016many} initiated the systematic study of this generalized problem and determined the order of magnitude of $\mathrm{ex}(n, K_{a,b}, K_{s,t})$ when $K_{a,b}$ is substantially smaller than $K_{s,t}$. Their results were later extended and refined in a number of subsequent works \cite{ma2020generalized, Bay, Ger22}. In particular, Ma, Yuan, and Zhang \cite{ma2020generalized} established the order of magnitude for all sufficiently large $t$ whenever $a < s$ and $b \le s$. The case $a < s \le b$ is trivial for all values of $t$, as observed in \cite{Ger22}. These results together cover all cases with $a \le s$ and  large enough $t$.

On the other hand, very little is known about the order of magnitude of  $\mathrm{ex}(n, K_{a,b}, K_{s,t})$ when $a > s$. For the problem to be nontrivial, one must have $s < a \le b < t$; otherwise the answer follows from elementary counting arguments. For this parameter range, a standard double-counting argument yields the universal upper bound
\begin{equation}\label{equation-upper_bound}
\mathrm{ex}(n, K_{a,b}, K_{s,t}) = O_{a,b,s,t}(n^s),
\end{equation}
first noted in \cite{Ger19} and independently observed in later works \cite{janzer2026generalized, pohoata2026k2t1free}. 

A natural question is whether the upper bound in (\ref{equation-upper_bound}) is tight. 
Recently, Pohoata, Tidor, and Yu~\cite{pohoata2026k2t1free} proved that for all $t \ge 3$,
\begin{equation*}
\ex(n, K_{t,t}, K_{2,t+1}) = \Theta_t(n^2),
\end{equation*}
thereby answering a question of Spiro. %\cite{aa}. 
Subsequently, Taranchuk \cite{Ta} used an explicit construction to show that when $t$ is a prime power and $n = t^{2e-1}$,
\begin{equation*}
\ex(n, K_{t,t}, K_{2,t+1}) = (1 + o(1))\frac{n^2}{2t(t - 1)}.
\end{equation*}

For the case $s=3$, Janzer, Longbrake, and Yepremyan~\cite{janzer2026generalized} showed that for any $3 < a \le b$, there exists $t_0$ such that for all $t \ge t_0$,
\begin{equation*}
\ex(n, K_{a,b}, K_{3,t}) = \Theta_{a,b,t}(n^3).
\end{equation*}
%Their proof is based on a projective point-hyperplane incidence construction in dimension five, together with the bounded-complexity variety lemma of Bukh and Conlon~\cite{bukh2018rational} to control plane sections. 
Very recently, Hou, Hu, and Wang~\cite{hou2026explicit} made this threshold explicit by showing that the conclusion holds for all $t\geq 2\max\{3,\lceil b/2\rceil\}+1$. In particular, for every even $b \geq 6$, this matches the necessary threshold $t = b+1\geq7$. However, the case where $b$ is odd remained open. Indeed, Hou, Hu, and Wang~\cite{hou2026explicit} explicitly list it as Problem~4.1. 

In this paper, we resolve the remaining case where $b\geq5$ is odd, thereby answering Problem~4.1 in \cite{hou2026explicit}. In fact, our argument works for all $b\ge 5$, independently of the parity of $b$. Thus,
apart from the case $b=4$, this also settles the $s=3$ case of Conjecture~6.1 in \cite{janzer2026generalized}.
Our main result is as follows. 

\begin{theorem} \label{thm:main}
Let integers $b \geq 5$ and $4 \leq a \leq b$. Then
\begin{equation*}
\ex(n, K_{a,b}, K_{3,b+1}) = \Theta_{a,b}(n^3). 
\end{equation*}
\end{theorem}

The upper bound is elementary and follows from a standard double-counting argument. The main difficulty lies in the lower bound, which we establish via an explicit finite-geometric construction. 

We now briefly describe the main ideas behind the construction. Our construction is carried out in the projective space $\PG(5,q)$, that is, the $5$-dimensional projective space over the finite field $\mathbb{F}_q$, where $q$ is an odd prime power. 
We introduce a point set $S_f \subset \PG(5, q)$ parametrized by three affine coordinates, with the remaining coordinates given by quadratic expressions and a univariate polynomial $f$ of degree $b$. The point set is designed such that: 
\begin{itemize}
  \item no three points are collinear (Lemma~\ref{lem:no-three-collinear});
  \item any plane not contained in a certain exceptional family intersects $S_f$ in at most $b$ points (Lemma~\ref{lem:plane-sections});
  \item  the remaining exceptional planes are excluded by a polarity argument (Lemma~\ref{lem:polarity-exceptional}). 
\end{itemize}

For a polynomial $f$ of degree $b$ and a polynomial  $\delta$ of degree less than $b$, we say $f+\delta$ is the \emph{shifted polynomial} of $f$.
The key lemma is a polynomial lemma (Lemma~\ref{lem:double-shift}), which guarantees the existence of many polynomials $f$ such that both $f$ and $f+\delta$ split completely into distinct linear factors, where $\delta$ is any fixed polynomial of degree less than $b$. 
This result follows from the theorem of Andrade, Bary-Soroker, and Rudnick on the independence of shifted factorization types~\cite{andrade2013shifted}. 
Using this lemma, we show that there exists a polynomial $f$ for which the incidence graph derived from $S_f$ contains $\Omega_b(q^9)$ copies of $K_{b,b}$. This yields $\Omega_{a,b}(n^3)$ copies of $K_{a,b}$, establishing the desired lower bound.

The rest of the paper is organized as follows. In Section~\ref{sec:polynomial-lemma}, we establish a polynomial lemma on simultaneous complete splittings of shifted polynomials. In Section~\ref{sec:point-set}, we define the finite-field point set, introduce the orthogonal polarity, and prove the basic geometric properties needed for the construction. More specifically, in Subsection~\ref{sec:no-three-collinear} we show that $S_f$ contains no three collinear points, and in Subsection~\ref{sec:plane-sections} we analyze the intersections of $S_f$ with projective planes. 
In Section~\ref{sec:construction}, we construct the $K_{3,b+1}$-free bipartite incidence graph and prove that for a suitable choice of $f$, it contains many copies of $K_{a,b}$. In Section~\ref{sec:proof}, we prove Theorem~\ref{thm:main}. Finally, we give some concluding remarks in Section~\ref{sec:concluding}.

%The rest of the paper is organized as follows. In Section~\ref{sec:polynomial-lemma}, we establishes the polynomial lemma on double-shift complete splittings. In Section~\ref{sec:point-set}, we define the finite-field point set, establish the polarity argument, and prove its basic geometric properties. In particular, we establish that $S_f$ contains no three collinear points in Subsection~\ref{sec:no-three-collinear}, and analyze the intersections of the point set with affine planes in Subsection~\ref{sec:plane-sections}. In Section~\ref{sec:construction}, we construct the $K_{3,b+1}$-free incidence graph. In Section \ref{sec:proof}, we give the proof of Theorem \ref{thm:main}. Finally, we state some concluding remarks in Section~\ref{sec:concluding}.

\section{A polynomial lemma over finite fields} \label{sec:polynomial-lemma}

In this section, we give a lower bound for a polynomial and its shifted polynomials simultaneously splitting completely over finite fields, which will be useful for our construction.

%Firstly, we introduce the following lemma about the cycle structure of shifted polynomials. 
%which is an immediate consequence of the finite-field shifted factorization theorem of Andrade, Bary-Soroker and Rudnick \cite{andrade2013shifted}, after excluding the $O_b(q^{b-1})$ nonsquarefree polynomials.
% that will be used to ensure the existence of many ``rich'' planes in our construction. %The lemma  asserts that, for any fixed polynomial $\delta$ of degree less than $b$, there are many monic degree-$b$ polynomials $F$ such that both $F$ and $F + \delta$ split completely into distinct linear factors.

Let $\mathcal{M}_b$ denote the set of monic polynomials of degree $b$ over $\F_q$. Then  $|\mathcal{M}_b| = q^b$.
For $f\in \mathcal{M}_b$, define the factorization type of $f$ by $\mathrm{type}(f)=(\nu_1(f),\nu_2(f),\ldots,\nu_b(f))$,  where $\nu_i(f)$ is the number of irreducible factors of $f$ of degree $i$. 
%Thus,
%$$
%\sum_{j=1}^d j\nu_j(f)=d.
%$$
For a partition $\mu=(\mu_1,\ldots,\mu_b)$ of $b$, let
$$
\chi_\mu(f)=
\begin{cases}
1, & \text{if } \operatorname{type}(f)=\mu,\\
0, & \text{otherwise},
\end{cases}
$$
and
$$
p(\mu)=\prod_{j=1}^b \frac{1}{j^{\mu_j}\mu_j!}.
$$
%Equivalently, \(p(\mu)\) is the probability that a uniformly random permutation in \(S_d\) has cycle type \(\mu\).
Let $\mathrm{deg}(f)$ denote the degree of the polynomial $f$. We firstly introduce the following result for shifted polynomials of $f$.
\begin{theorem}[Andrade, Bary-Soroker and Rudnick, \cite{andrade2013shifted}]\label{ABR-th}
For fixed positive integers $b$ and $s$, we have
\[
\frac{1}{q^b} \sum_{f \in \mathcal{M}_b} \chi_{\mu_1}(f+h_1) \cdots \chi_{\mu_s}(f+h_s) = p(\mu_1) \cdots p(\mu_s) + O_b(q^{-1/2}),
\]
uniformly for all partitions $\mu_1, \dots, \mu_s$ and distinct polynomials $h_1, \dots, h_s \in \mathbb{F}_q[t]$ with $\deg(h_i) < b$ as $q \to \infty$.
\end{theorem}

Let $s = 2$, and $\mu_1 = \mu_2=(b,0,\ldots,0)$ in Theorem \ref{ABR-th}.
Then $\chi_{\mu_i}(f)$ is $1$ if and only if $f$ splits into $b$ linear factors over $\mathbb{F}_q$ (Here we allow multiple roots). And $p(\mu_1)=p(\mu_2)= \frac{1}{b!}.$
Let $h_1 = 0$ and $h_2 = \delta\neq 0$ with $\deg \delta < b$.
Applying Theorem \ref{ABR-th} gives
\[
\frac{1}{q^b} \sum_{f \in \mathcal{M}_b} \chi_{\mu_1}(f) \, \chi_{\mu_2}(f+\delta) = \frac{1}{(b!)^2} + O_b(q^{-1/2}).
\]
Multiplying both sides by $q^b$ , we get the following result.
\begin{corollary}\label{f and f+delta split}
    Let $\delta\neq 0$. The number of $f \in \mathcal{M}_b$ such that $f$ and $f+\delta$ split into $b$ linear factors is 
\begin{equation*}\label{f and f+delta split}
    \frac{q^b}{(b!)^2} + O_b\left(q^{b-1/2}\right),
\end{equation*}
%which is uniform for all $\delta$ with $\deg \delta < b$. 
\end{corollary}

%\[
%\#\left\{f \in \mathcal{M}_b : f \text{ and } f+\delta \text{ split into $b$ linear factors} \right\} = \frac{q^b}{(b!)^2} + O_b\left(q^{b-1/2}\right),
%\]

%This ensures that our final constants $q_0$ and $\eta_b$ will not depend on $\delta$.

Next we give the lower bound for the number of $f \in \mathcal{M}_b$ with $f$ and $f+\delta$ split into $b$ distinct linear factors for sufficiently large $q$.

\begin{lemma} \label{lem:double-shift}
Let $b \geq 5$ be an integer. There exist constants $q_0=q_0(b)$ and $\eta_b > 0$ such 
that for $q \geq q_0$ and polynomial 
$\delta \in \F_q[t]$ with $\deg \delta < b$, there are at least $\eta_b q^b$ polynomials 
$f \in \mathcal{M}_b$ such that both $f$ and $f + \delta$ split into $b$ 
distinct linear factors over $\F_q$.
\end{lemma}

\begin{proof}

%which is strictly larger than the leading term above, so the lower bound will hold automatically.

If $\delta\neq0$, by Corollary \ref{f and f+delta split}, then we only need to consider the multiple roots of $f$ and $f+\delta$.
The polynomial $f$ is said to be \emph{not square-free} if there exists a monic irreducible  polynomial $P$ such that $P^2 \mid f$. The number of not square-free polynomials $f\in \mathcal{M}_b$ is at most
%We first assert that the total number of non-square-free monic degree-$b$ polynomials over $\mathbb{F}_q$ is $O_b(q^{b-1})$. 
%Indeed, if $F$ is not square-free, there exists a monic irreducible polynomial $P$ such that $P^2 \mid F$. Let $r = \deg P$, so $2r \leq b$. There are at most $q^r$ monic irreducible polynomials of degree $r$, and at most $q^{b-2r}$ choices for the cofactor $F/P^2$. Summing over $r = 1, 2, \dots, \lfloor b/2 \rfloor$, we have
\[
\sum_{r=1}^{\lfloor b/2 \rfloor}q^rq^{b-2r}=\sum_{r=1}^{\lfloor b/2 \rfloor}q^{b-r}=O_b(q^{b-1}).
\]
%which  yields a total of $O_b(q^{b-1})$ non-square-free polynomials.

Note that the map $f \mapsto f+\delta$ is a bijection on $\mathcal{M}_b$.
%: since $\deg \delta < b$, the leading term $T^b$ is preserved, and the inverse map is $F \mapsto F-\delta$. 
Therefore, the number of $f$ for which $f+\delta$ is not square-free is also $O_b(q^{b-1})$.
By Corollary \ref{f and f+delta split} and deleting all not square-free polynomials $f$ and $f+\delta$, we conclude that  the number of $f \in \mathcal{M}_b$ such that $f$ and $f+\delta$ split into $b$ distinct linear factors is 
$$\frac{q^b}{(b!)^2} + O_b\left(q^{b-1/2}\right).$$
%By the union bound, the number of polynomials $F$ for which at least one of $F$ or $F+\delta$ has a repeated root is at most
%$$O_b(q^{b-1}) + O_b(q^{b-1}) = O_b(q^{b-1}).$$
%Since $q^{b-1}$ is of lower order than $q^{b-1/2}$, subtracting these bad polynomials does not affect the leading term. We conclude that 
%\[
%\#\left\{ F \in \mathcal{M}_b : F, F+\delta \text{ split into distinct linear factors} \right\} = \frac{q^b}{(b!)^2} + O_b\left(q^{b-1/2}\right).
%\]
By the definition of $O_b\left(q^{b-1/2}\right)$, there exists a constant $C_b > 0$ such that
$$\left| O_b\left(q^{b-1/2}\right) \right| \leq C_b \cdot q^{b-1/2}.$$
Thus, we have
\begin{equation}\label{exKabK3b+1-equation1}
\frac{q^b}{(b!)^2} + O_b\left(q^{b-1/2}\right) \geq \frac{q^b}{(b!)^2} - C_b \, q^{b-1/2} = q^b \left( \frac{1}{(b!)^2} - \frac{C_b}{\sqrt{q}} \right).
\end{equation}
Choose $q_0 = q_0(b)$ sufficiently large so that for all $q \geq q_0$,
\begin{equation}\label{exKabK3b+1-equation2}
\frac{C_b}{\sqrt{q}} \leq \frac{3}{4(b!)^2}.
\end{equation}
Fix $\eta_b=\frac{1}{4(b!)^2}$.  Substituting (\ref{exKabK3b+1-equation2}) into (\ref{exKabK3b+1-equation1}), we get
$$\frac{q^b}{(b!)^2} + O_b\left(q^{b-1/2}\right) \geq \eta_bq^b.$$ 

If $\delta = 0$, then the number of $f \in \mathcal{M}_b$ such that $f$ and $f+\delta$ split into $b$ distinct linear factors is 
$$\binom{q}{b} = \frac{q^b}{b!} + O_b(q^{b-1}).$$
Similar to the discussion for $\delta\neq 0$, we get the desired result.
\end{proof}

\section{The algebraic point set and orthogonal polarity}
\label{sec:point-set}

%Then the quadratic form
%$
%r^2 - \alpha s^2=0
%$
%if and only if $r = s = 0$.

Let $\PG(5,q)$ be the $5$-dimensional projective space over $\mathbb{F}_q$, whose points are represented by homogeneous coordinates $[w:x:y:z:u:v]$. The affine chart defined by $w = 1$ is naturally identified with the 5-dimensional affine space $\mathbb{A}^5$ with coordinates $(x,y,z,u,v)$.

Let $\alpha \in \mathbb{F}_q$ be a fixed nonsquare element and $f \in \mathcal{M}_b$. We define the point set
\[
S_f = \left\{ s_f(x, y, z) : x, y, z \in \mathbb{F}_q \right\},
\]
where 
\begin{equation}\label{exKabK3b+1-equation3}
    s_f(x, y, z) = \left[ 1 : x : y : z : x^2 : y^2 - \alpha z^2 + f(x) \right].
\end{equation}
Since $(x,y,z)$ uniquely determine each point, the parametrization is injective. Thus
\[
|S_f| = q^3.
\]

Let $p_1=(w_{p_1},x_{p_1},y_{p_1},z_{p_1},u_{p_1},v_{p_1})$ and $p_2=(w_{p_2},x_{p_2},y_{p_2},z_{p_2},u_{p_2},v_{p_2})$ be two points of $\mathbb{F}_q^6$.
We define the symmetric bilinear form $\beta$ over $\mathbb{F}_q^6$ by 
\[
\begin{aligned}
\beta(p_1, p_2) = \, & w_{p_1} y_{p_2} + y_{p_1} w_{p_2} + x_{p_1} v_{p_2} + v_{p_1} x_{p_2} +z_{p_1} z_{p_2} + u_{p_1} u_{p_2}.
\end{aligned}
\]
Reordering the basis $(w,y), (x,v), z, u$, then the matrix of $\beta$ is
$$
\operatorname{diag}\left(\begin{pmatrix}0&1\\1&0\end{pmatrix},\begin{pmatrix}0&1\\1&0\end{pmatrix},1,1\right),
$$
which is invertible. Hence, $\beta$ is nondegenerate.

For any projective subspace $\Lambda\subseteq \PG(5,q)$, let $\Lambda^\perp$ denote its polar subspace with respect to $\beta$. Since $\beta$ is nondegenerate, we have
\[
\dim \Lambda + \dim \Lambda^\perp = 4,
\]
and $(\Lambda^\perp)^\perp = \Lambda$. In particular, the polar of a projective plane is also a projective plane.

For $c \in \mathbb{F}_q$, define the 3-dimensional projective subspace of $\PG(5,q)$
\begin{equation}\label{exKabK3b+1-equation4}
F_c = \{ x = c w,\ u = c^2 w \}.
\end{equation}
For the points $(x, y, z, u, v)$ in the affine chart $w=1$, define the projection
\[
\rho : \mathbb{A}^5 \to \mathbb{A}^2, \quad \rho(x, y, z, u, v) = (x, u).
\]

\subsection{No three collinear points}
\label{sec:no-three-collinear}

In this subsection, we prove that $S_f$ contains no three collinear points, which implies that any three points of $S_f$ span a plane.

\begin{lemma} \label{lem:no-three-collinear}
For every $f \in \mathcal{M}_b$, the set $S_f$ contains no three distinct collinear points.
\end{lemma}

\begin{proof}
Suppose that there exist three distinct collinear points in $S_f$:
%Since all points lie in the affine chart $w=1$, suppose that these three points are
\[
s_f(p), \quad s_f(p + h), \quad s_f(p + \lambda h),
\]
where $p = (x, y, z)$, $h = (r, s, u) \neq (0, 0, 0)$, and $\lambda \in \mathbb{F}_q \setminus \{0, 1\}$. Then
\[
s_f(p + \lambda h) = (1 - \lambda) s_f(p) + \lambda s_f(p + h).
\]
Combining with (\ref{exKabK3b+1-equation3}), we get
\[
(x + \lambda r)^2 = (1 - \lambda) x^2 + \lambda (x + r)^2,
\]
which implies that
\[
\lambda(\lambda - 1) r^2 = 0.
\]
Since $\lambda \neq 0, 1$, we have $r = 0$.
Thus, the $x$-coordinate is constant across all three points. So $f(x)$ is also constant. Now we compare the $v$-coordinates, 
\[
(y + \lambda s)^2 - \alpha (z + \lambda u)^2 = (1 - \lambda)(y^2 - \alpha z^2) + \lambda ((y + s)^2 - \alpha (z + u)^2),
\]
which implies that
\[
\lambda(\lambda - 1) (s^2 - \alpha u^2) = 0.
\]
Since $\lambda \neq 0, 1$, we obtain
\[
s^2 - \alpha u^2 = 0.
\]
Since $\alpha$ is a fixed nonsquare element, we have $s = u = 0$. Thus, $h = (0,0,0)$, contradicting $h \neq 0$.
Therefore, no three distinct points of $S_f$ are collinear.
\end{proof}

\subsection{Plane sections}
\label{sec:plane-sections}

%In this subsection, we obtain two bounds on the size of intersections between $S_f$ and projective planes: generic planes intersect $S_f$ in at most $b$ points (see Lemma \ref{lem:plane-sections}), while planes containing an exceptional line intersect $S_f$ in at most $2$ points (see Lemma \ref{lem:polarity-exceptional}).

In this subsection, we obtain the upper bounds on the size of the intersection of $S_f$ with two families of projective
planes, planes that are not contained in any $F_c$ (as shown in (\ref{exKabK3b+1-equation4}), and  planes that contain the polar of $F_c$.
\begin{lemma} \label{lem:plane-sections}
Let $b \geq 5$  and let $\Pi$ be a projective plane not contained in $F_c$ for any $c \in \mathbb{F}_q$. Then
\[
|S_f \cap \Pi| \leq b.
\]
\end{lemma}

\begin{proof}
If $\Pi \subseteq \{w = 0\}$, then $S_f \cap \Pi = \emptyset$ since every point of $S_f$ has $w = 1$. Then the bound holds. Otherwise, let
$
H = \Pi \cap \{w = 1\}
$
be the affine part of the plane. We consider the following cases.

\medskip
\noindent \textbf{Case 1. $\dim \rho(H) = 2$.}

Then $\rho|_H$ is an affine isomorphism onto its image in $\mathbb{A}^2$, so there exist affine linear functions $L_1, L_2, L_3 : \mathbb{A}^2 \to \mathbb{A}^1$ such that
\[
H = \left\{ (x, L_1(x, u), L_2(x, u), u, L_3(x, u)) : x, u \in \mathbb{F}_q \right\}.
\]
Write
\[
\begin{aligned}
L_1(x, u) &= A_0 + A_1 x + A_2 u, \\
L_2(x, u) &= B_0 + B_1 x + B_2 u, \\
L_3(x, u) &= C_0 + C_1 x + C_2 u.
\end{aligned}
\]
A point $s_f(x,y,z) \in S_f$ lies in $H$ if and only if $u = x^2$ and the $v$-coordinate matches. Substituting $u = x^2$ into the linear forms gives
\begin{equation}\label{exKabK3b+1-equation5}
f(x) + (A_0 + A_1 x + A_2 x^2)^2 - \alpha (B_0 + B_1 x + B_2 x^2)^2 - (C_0 + C_1 x + C_2 x^2)=0.
\end{equation}
All terms other than $f(x)$ have degree at most $4$. Since $b \geq 5$ and $f\in \mathcal{M}_b$, the left-hand side of (\ref{exKabK3b+1-equation5}) is also a   polynomial of degree $b$ in $x$, which has at most $b$ roots. Each root uniquely determines a point in $H \cap S_f$, so $|H \cap S_f| \leq b$.

\medskip
\noindent \textbf{Case 2. $\dim \rho(H) = 1$.}

Then $\rho(H)$ is an affine line in the $(x,u)$-affine plane. This line intersects the parabola defined by $u = x^2$ in at most two points, which vertical lines intersect the parabola in at most one point and non-vertical lines yield a monic quadratic equation in $x$.
Fix an intersection point $(c, c^2) \in \rho(H)$. The 
\[
H \cap \rho^{-1}(c, c^2)
\]
is an affine line in $(y,z,v)$-space, written as
\[
(y, z, v) = (y_0, z_0, v_0) + t(r, s, h), \quad t \in \mathbb{F}_q.
\]
A point of $S_f$ lies on this line if and only if
\[
v_0 + t h = (y_0 + t r)^2 - \alpha (z_0 + t s)^2 + f(c).
\]
If $(r,s) \neq (0,0)$, then the coefficient of $t^2$ is $r^2 - \alpha s^2 \neq 0$. Thus the equation is a nonzero quadratic in $t$ with at most two solutions. If $r = s = 0$, then $h \neq 0$ and the equation is a nonzero linear equation in $t$ with exactly one solution.
Thus each intersection $(c,c^2)$ contributes at most two points to $H \cap S_f$. Since there are at most two such intersections and $b \geq 5$, we have
\[
|H \cap S_f| \leq 4 \leq b.
\]

\medskip
\noindent \textbf{Case 3. $\dim \rho(H) = 0$.}

Then there exist constants $c,d \in \mathbb{F}_q$ such that $x = c$ and $u = d$ identically on $H$. If $d \neq c^2$, then $H \cap S_f = \emptyset$ since every point of $S_f$ satisfies $u = x^2$. If $d = c^2$, then $\Pi$ is contained in $F_c$, a contradiction.
\end{proof}

%We now analyze planes that lie inside some exceptional subspace $F_c$. We first compute the polar of $F_c$.
For $c \in \mathbb{F}_q$, define the projective line $K_c$ by
$$
K_c:=\left\{[w:x:y:z:u:v]\in\PG(5,q): w=x=z=0, y+cv+c^2u=0\right\}
$$
%Define the 2-dimensional vector subspace $\widetilde{K}_c$ by
%\[
%\widetilde{K}_c: = \spanop\left\{ (0, 0, -c, 0, 0, 1), \, (0, 0, -c^2, 0, 1, 0) \right\},
%\]
%which corresponds to a projective line.

\begin{lemma} \label{lem:polar-of-Fc}
Let $F_c$ be defined as shown in (\ref{exKabK3b+1-equation4}). Then 
$
F_c^\perp = K_c.
$
\end{lemma}

\begin{proof}
Suppose that $p = [w_p:x_p: y_p: z_p: u_p: v_p]\in F_c^\perp$. Then for every $[w:cw:y:z:c^2w:v]\in F_c$, we have
\[
w_p y + x_p v + z_p z + (y_p +  v_pc +  u_pc^2) w = 0,
\]
which holds  if and only if $w_p = x_p = z_p = 0$ and $y_p +  v_pc +  u_pc^2 = 0$. The set of such points is precisely $K_c$.
\end{proof}

\begin{lemma} \label{lem:polarity-exceptional}
Let $\Lambda$ be a projective plane containing $K_c$. Then
\[
|S_f \cap \Lambda| \leq 2.
\]
\end{lemma}

\begin{proof}
Define $\widetilde{K}_c$ by
$$
\widetilde{K}_c= \spanop\left\{ (0, 0, -c, 0, 0, 1), \, (0, 0, -c^2, 0, 1, 0) \right\},
$$
which is the 2-dimensional vector subspace corresponding to $K_c$.
Consider the linear map $L_c : \mathbb{F}_q^6 \to \mathbb{F}_q^4$ defined by
\[
L_c(w, x, y, z, u, v) = (w, x, z, y+cv+c^2u).
\]
The kernel of $L_c$ is given by $w = x = z = 0$ and $y + c v + c^2 u = 0$, which is exactly the $\widetilde{K}_c$.

Since $\Lambda$ is a projective plane containing $K_c$, there is a 3-dimensional vector subspace $\tilde{\Lambda}$ contains $\widetilde{K}_c$. Then
\[
\dim L_c(\tilde{\Lambda}) = \dim \tilde{\Lambda} - \dim \widetilde{K}_c = 3 - 2 = 1.
\]
Thus $L_c(\tilde{\Lambda})$ is a 1-dimensional subspace. Therefore, for all points in $S_f \cap \Lambda$, the
$
x, z, \ell := y + c v + c^2 u
$
are constants since all points of $S_f$ have $w = 1$.

For any point $s_f(x,y,z) \in S_f$, we have
\[
u = x^2, \quad v = y^2 - \alpha z^2 + f(x).
\]
Substituting into the expression for $\ell$ gives
\[
\ell = y + c(y^2 - \alpha z^2 + f(x)) + c^2 x^2.
\]
Then, we have
%Since $x$, $z$, and $\ell$ are fixed, $y$ satisfies the equation
\begin{equation}\label{exKabK3b+1-equation6}
c y^2 + y + c(-\alpha z^2 + f(x)) + c^2 x^2 - \ell = 0.
\end{equation}
Since $x$, $z$, and $\ell$ are fixed,  the left hand of (\ref{exKabK3b+1-equation6}) is the polynomial in $y$. If $c = 0$, then the linear term has coefficient $1$. If $c \neq 0$, then the quadratic term has coefficient $c \neq 0$. Hence there are at most two solutions for $y$, each determining a unique point of $S_f$. Therefore,
\[
|S_f \cap \Lambda| \leq 2. \qedhere
\]
\end{proof}

\section{The bipartite incidence graph}
\label{sec:construction}

Let $G_f$ be the bipartite graph with vertex partition $V(G_f) = L \cup R$, where $L$ and $R$ are disjoint copies of $S_f$. A vertex $p_L \in L$ is adjacent to a vertex $q_R \in R$ if and only if $\beta(p_L, q_R) = 0$.

\begin{lemma} \label{lem:k3t-free}
For every $f \in \mathcal{M}_b$, the graph $G_f$ is $K_{3,b+1}$-free.
\end{lemma}

\begin{proof}
Take any three distinct vertices $p_1, p_2, p_3 \in L$. By Lemma~\ref{lem:no-three-collinear}, the vertices $p_1, p_2, p_3$ are not collinear, so they span a unique projective plane
\[
\Pi = \langle p_1, p_2, p_3 \rangle.
\]
By definition of adjacency, the common neighbors of $p_1,p_2,p_3$ in $R$ are exactly the points in $S_f \cap \Pi^\perp$.

If $\Pi^\perp$ is not contained in any $F_c$, then by Lemma~\ref{lem:plane-sections}, we have
\[
|S_f \cap \Pi^\perp| \leq b,
\]
which implies that the three vertices have at most $b$ common neighbors.

So we assume that $\Pi^\perp \subseteq F_c$ for some $c$. By polarity reverses inclusion, we have
\begin{align}\label{eq6}
F_c^\perp \subseteq (\Pi^\perp)^\perp = \Pi.
\end{align}
Notice that $F_c^\perp = K_c$ by Lemma~\ref{lem:polar-of-Fc}. Thus by (\ref{eq6}), we have $K_c \subseteq \Pi$. Now we apply  Lemma~\ref{lem:polarity-exceptional} to conclude that
\[
|S_f \cap \Pi| \leq 2,
\]
which contradicts the fact that $p_1,p_2,p_3 \in S_f \cap \Pi$. Therefore this case cannot occur.

By symmetry of the bilinear form $\beta$, the adjacency relation in $G_f$ is symmetric with respect to the two vertex partitions $L$ and $R$. And  $\beta(p,q)=\beta(q,p)$ implies that the graph construction treats the  vertex sets $L$ and $R$ identically. So we can get the same result for any three vertices in $R$ by the similar proof. Therefore, $G_f$ contains no copy of $K_{3,b+1}$. 
%This completes the proof of Lemma \ref{lem:k3t-free}.
\end{proof}

It remains to show that there exists a polynomial $f$ for which $G_f$ contains many copies of $K_{a,b}$. We achieve this by constructing a large family of planes, where each such plane induces a  subgraph $K_{b,b}$ in $G_f$.

\begin{lemma} \label{lem:good-plane-count}
Fix integers $b \geq 5$ and $4 \leq a \leq b$. For all sufficiently large odd prime powers $q$, there exists a $f \in \mathcal{M}_b$ such that  the bipartite graph $G_f$ has $2q^3$ vertices and $\Omega_{a,b}(q^9)$ copies of $K_{a,b}$.
\end{lemma}

\begin{proof}
Let $\PG(5,q)$ be a 5-dimensional projective space  over $\mathbb{F}_q$  with the nondegenerate symmetric bilinear form $\beta$ that induces an orthogonal polarity. Recall the point set
\[
S_f = \left\{ s_f(x, y, z) : x, y, z \in \mathbb{F}_q \right\},
\]
where 
\[
s_f(x, y, z) = \left[ 1 : x : y : z : x^2 : y^2 - \alpha z^2 + f(x) \right],
\]
with $\alpha \in \mathbb{F}_q$ a fixed non-square and $f \in \mathcal{M}_b$. By the definition of  $G_f$, $G_f$ has exactly $2|S_f| = 2q^3$ vertices.

We first introduce a 9-parameter family of projective planes. 
For any triple of polynomials of degree at most $2$:
\[
\begin{aligned}
A(t) &= a_0 + a_1 t + a_2 t^2, \\
B(t) &= b_0 + b_1 t + b_2 t^2, \\
C(t) &= c_0 + c_1 t + c_2 t^2,
\end{aligned}
\]
we define the projective plane $\Pi(A,B,C)$ as the set of points satisfying the linear system
\[
\begin{aligned}\label{exKabK3b+1-equation7}
y &= a_0 w + a_1 x + a_2 u, \\
z &= b_0 w + b_1 x + b_2 u, \\
v &= c_0 w + c_1 x + c_2 u.
\end{aligned}
\]
Each equation is a homogeneous linear relation, and the three equations are linearly independent (each isolates a distinct coordinate $y,z,v$ with leading coefficient $1$). The solution space is therefore a $3$-dimensional vector subspace, i.e., a projective plane in $\PG(5,q)$.
Distinct triples $(A,B,C)$ yield distinct planes. Indeed, if two triples produce the same plane, the linear expressions for $y,z,v$ in terms of $w,x,u$ must coincide termwise, so all coefficients $a_i,b_i,c_i$ must be equal. Thus, the full family therefore contains $q^3 \cdot q^3 \cdot q^3 = q^9$ distinct planes.

We now establish a bijection between $S_f \cap \Pi(A,B,C)$ and the roots of a polynomial over $\mathbb{F}_q$. A point $s_f(x,y,z) \in S_f$ belongs to $\Pi(A,B,C)$ if and only if its coordinates satisfy the three plane equations. Substituting $w=1$, $u=x^2$, $v = y^2 - \alpha z^2 + f(x)$ into (\ref{exKabK3b+1-equation7}), we get
\[
\begin{aligned}
~~~~~~~~~~~~~~~~~&y = a_0 + a_1 x + a_2 x^2 = A(x), \\
&z= b_0 + b_1 x + b_2 x^2 = B(x), \\
&y^2 - \alpha z^2 + f(x) = c_0 + c_1 x + c_2 x^2 = C(x).
\end{aligned}
\]
Substitute $y = A(x)$ and $z = B(x)$ into $y^2 - \alpha z^2 + f(x) = C(x)$, we have
\[
f(x) + A(x)^2 - \alpha B(x)^2 - C(x) = 0.
\]
Define %the \emph{section polynomial}
\[
P_{\Pi,f}(t) = f(t) + G_\Pi(t), \quad \text{where } G_\Pi(t) = A(t)^2 - \alpha B(t)^2 - C(t).
\]
Then $s_f(x,y,z) \in \Pi$ if and only if $x$ is a root of $P_{\Pi,f}$ over $\mathbb{F}_q$. Moreover, each root $x$ uniquely determines $y=A(x)$ and $z=B(x)$, and then uniquely determines a point in $S_f \cap \Pi$. This gives a bijection between $S_f \cap \Pi$ and the set of distinct roots of $P_{\Pi,f}$ over $\mathbb{F}_q$. In particular, $|S_f \cap \Pi|$ equals the number of distinct roots of $P_{\Pi,f}$ over $\mathbb{F}_q$.

Since $\deg A(t) \leq 2$ and $\deg B(t) \leq 2$, we have $\deg(A(t)^2) \leq 4$ and $\deg(B(t)^2) \leq 4$. So $\deg G_\Pi \leq 4$. By assumption $b \geq 5 > 4$,  the leading term of $P_{\Pi,f}$ is identical to that of $f$. Thus $P_{\Pi,f}$ is always a monic polynomial of degree $b$.

The following claim shows that the subfamily of planes with $b_2 \neq 0$ is closed under the orthogonal polarity.

\begin{claim}\label{clm:polar-closure}
If $b_2 \neq 0$, then $\Pi(A,B,C)^\perp = \Pi(A',B',C')$ for some  polynomials $A',B',C'$ of degree at most $2$.
\end{claim}

\begin{proof}[Proof of Claim \ref{clm:polar-closure}]
Consider the following three points in $\mathbb{F}_q^6$, 
\[
\begin{aligned}
r_A &= (1, 0, -a_0, 0, -a_2, -a_1), \\
r_B &= (0, 0, -b_0, 1, -b_2, -b_1), \\
r_C &= (0, 1, -c_0, 0, -c_2, -c_1).
\end{aligned}
\]
We first verify that each point is orthogonal to every point in $\Pi(A,B,C)$. Take arbitrary $p = (w_p,x_p,y_p,z_p,u_p,v_p) \in \Pi(A,B,C)$. Then
\[
y_p = a_0 w_p + a_1 x_p + a_2 u_p.
\]
Thus, we get
\begin{align*}
\beta(r_A,p)
&= w_{r_A}y_p + y_{r_A}w_p + x_{r_A}v_p + v_{r_A}x_p + z_{r_A}z_p + u_{r_A}u_p \\
&= 1\cdot y_p + (-a_0)\cdot w_p + 0\cdot v_p + (-a_1)\cdot x_p + 0\cdot z_p + (-a_2)\cdot u_p \\
&= y_p - a_0 w_p - a_1 x_p - a_2 u_p\\
&=0
\end{align*}
%Combined with the plane relation $y_p = a_0 w_p + a_1 x_p + a_2 u_p$, we get
%\[
%\beta(r_A,P) = 0.
%\]
Since $p$ is arbitrary in $\Pi(A,B,C)$, we conclude that $r_A \perp \Pi(A,B,C)$. 
The orthogonality of $r_B$ and $r_C$ follows by identical reasoning, using the defining equations for $z_P$ and $v_P$ respectively.

The vectors $r_A, r_B, r_C$ are also linearly independent. Indeed, each has a unique coordinate equal to $1$ (the $w$-coordinate for $r_A$, the $z$-coordinate for $r_B$, and the $x$-coordinate for $r_C$), so no nontrivial linear combination can yield the zero vector. Since $\beta$ is nondegenerate, the orthogonal complement of a $3$-dimensional subspace is also $3$-dimensional. Therefore $\{r_A, r_B, r_C\}$ forms a basis for $\Pi(A,B,C)^\perp$.

Any vector in $\Pi(A,B,C)^\perp$ can be written as $\lambda_1 r_A + \lambda_2 r_B + \lambda_3 r_C$ with $\lambda,\lambda_2,\lambda_3 \in \mathbb{F}_q$ and coordinates
\[
\begin{array}{lll}
w = \lambda_1, &\quad x = \lambda_3, &\quad y = -a_0\lambda_1 - b_0\lambda_2 - c_0\lambda_3, \\
z = \lambda_2, &\quad u = -a_2\lambda_1 - b_2\lambda_2 - c_2\lambda_3, &\quad v = -a_1\lambda_1 - b_1\lambda_2- c_1\lambda_3.
\end{array}
\]
We now express $y,z,v$ as linear functions of $w,x,u$. From the coordinate equations, immediately $\lambda_1 = w$ and $\lambda_3 = x$. 
Substitute these into the expression for $u$, we get that 
\begin{align}\label{eq11}
u = -a_2 w - b_2 \lambda_2 - c_2 x.
\end{align}
Since $b_2 \neq 0$, by (\ref{eq11}), we obtain that 
\begin{align}\label{eq12}
\lambda_2 = -\frac{a_2}{b_2} w - \frac{c_2}{b_2} x - \frac{1}{b_2} u.
\end{align}
Substitute $\lambda_1=w$, $\lambda_3=x$, and the equality (\ref{eq12}) into the formulas for $z$, $y$, and $v$, we obtain that
\[
\begin{aligned}
z &= \left(-\frac{a_2}{b_2}\right) w + \left(-\frac{c_2}{b_2}\right) x + \left(-\frac{1}{b_2}\right) u, \\
y&= \left(-a_0 + \frac{a_2 b_0}{b_2}\right) w
+ \left(-c_0 + \frac{b_0 c_2}{b_2}\right) x
+ \frac{b_0}{b_2} u, \\
v&= \left(-a_1 + \frac{a_2 b_1}{b_2}\right) w
+ \left(-c_1 + \frac{b_1 c_2}{b_2}\right) x
+ \frac{b_1}{b_2} u.
\end{aligned}
\]
All three coordinates $y,z,v$ are linear combinations of $w,x,u$ with constant coefficients. Defining $A',B',C'$ to be the  polynomials of degree at most $2$ with these coefficients, we conclude $\Pi(A,B,C)^\perp = \Pi(A',B',C')$.
\end{proof}

We call a plane $\Pi = \Pi(A,B,C)$ with $b_2 \neq 0$ \emph{good} if both $P_{\Pi,f}$ and $P_{\Pi^\perp,f}$ split completely into distinct linear factors over $\mathbb{F}_q$. Since both are monic polynomials of degree $b$, this condition is equivalent to $|S_f \cap \Pi| = b$ and $|S_f \cap \Pi^\perp| = b$.

Fix an arbitrary plane $\Pi$ in the subfamily. By Claim~\ref{clm:polar-closure}, $\Pi^\perp = \Pi(A',B',C')$ for some $A',B',C'$. Define
\[
G_{\Pi^\perp}(t) = A'(t)^2 - \alpha B'(t)^2 - C'(t),
\]
which satisfies $\deg G_{\Pi^\perp} \leq 4$ by the same degree argument. Let
\[
\delta(t) = G_{\Pi^\perp}(t) - G_\Pi(t).
\]
Then $\deg \delta \leq 4 < b$.

The condition that both $f + G_\Pi$ and $f + G_{\Pi^\perp}$ split completely into distinct linear factors is equivalent to the condition that both $F$ and $F+\delta$ split completely, where we set $F = f + G_\Pi$. Indeed, substituting $G_{\Pi^\perp} = G_\Pi + \delta$ gives
\[
f + G_{\Pi^\perp} = f + G_\Pi + \delta = F + \delta.
\]

We now assert that the map $\Phi: \mathcal{M}_b \to \mathcal{M}_b$ given by $\Phi(f) = f + G_\Pi$ is a bijection.
Since $\deg G_\Pi \leq 4 < b$, adding $G_\Pi$ does not change the leading term of $f$, so $F = f+G_\Pi$ remains monic of degree $b$.
If $\Phi(f_1) = \Phi(f_2)$, then $f_1 + G_\Pi = f_2 + G_\Pi$, so $f_1 = f_2$.
For any $F_0 \in \mathcal{M}_b$, set $f_0 = F_0 - G_\Pi$. 
Again $\deg G_\Pi < b$ implies that $f_0$ is monic of degree $b$, so $f_0 \in \mathcal{M}_b$ and $\Phi(f_0) = F_0$. 
Thus, the assertion holds.

By Lemma~\ref{lem:double-shift}, for fixed $\delta$ with $\deg \delta < b$, the number of $F \in \mathcal{M}_b$ for which both $F$ and $F+\delta$ split completely into distinct linear factors is at least $\eta_b q^b$, where $\eta_b > 0$ depends only on $b$. By the bijection $\Phi$, the number of $f \in \mathcal{M}_b$ making $\Pi$ good is also at least $\eta_b q^b$.
Since $|\mathcal{M}_b| = q^b$, if $f$ is chosen uniformly at random from $\mathcal{M}_b$, the probability that $\Pi$ is good is at least
$\frac{\eta_b q^b}{q^b} = \eta_b.$
This lower bound holds uniformly for every plane in the subfamily.

We next present the expected number of good planes.
Let $X_f$ denote the total number of good planes for a given $f$. 
We count the size of the $b_2 \neq 0$ subfamily:
$A(t)$ has $q^3$ choices (three free coefficients);
$B(t)$ has $q^2(q-1)$ choices, where $b_2 \in \mathbb{F}_q^*$ gives $q-1$ options,
and $b_0,b_1$ each have $q$ options;
and $C(t)$ has $q^3$ choices.
So the total number of planes in the subfamily is
\[
q^3 \cdot q^2(q-1) \cdot q^3 = (q-1) q^8.
\]
By linearity of expectation,
\[
\mathbb{E}_f [X_f] \geq \eta_b \cdot (q-1) q^8.
\]
Since the expectation meets this lower bound, there exists some $f \in \mathcal{M}_b$ for which
\[
X_f \geq \eta_b (q-1) q^8 = \Omega_b(q^9).
\]
That is, there exists an $f$ such that the number of good planes is $\Omega_b(q^9)$.

In the rest of proof, we  convert the count of good planes into a lower bound on the number of $K_{a,b}$ copies in $G_f$.
Let $\Pi$ be a good plane. Then $|S_f \cap \Pi| = b$ and $|S_f \cap \Pi^\perp| = b$. By definition of the orthogonal complement, every point in $\Pi$ is orthogonal to every point in $\Pi^\perp$. Therefore, the left vertices of $L$ corresponding to $S_f \cap \Pi$ and the right vertices of $R$ corresponding to $S_f \cap \Pi^\perp$ induce a complete bipartite subgraph $K_{b,b}$ in $G_f$.

Distinct good planes yield distinct $K_{b,b}$ subgraphs. Suppose two good planes $\Pi_1$ and $\Pi_2$ induce the same left vertex set $L = S_f \cap \Pi_1 = S_f \cap \Pi_2$. Since $b \geq 5$, we may choose any three points from $L$. By Lemma~\ref{lem:no-three-collinear}, no three points of $S_f$ are collinear, so three points uniquely determine a projective plane. Hence $\Pi_1 = \Pi_2$. Thus each good plane corresponds to a unique $K_{b,b}$.

Each such $K_{b,b}$ contains $\binom{b}{a}$ distinct copies of $K_{a,b}$, obtained by selecting any $a$ vertices from the left partition $L$ and all $b$ vertices from the right partition $R$. Since $a \geq 4$, any $K_{a,b}$ obtained this way has at least $3$ left vertices, which uniquely determine the underlying good plane. Therefore,  the number of $K_{a,b}$ copies in $G_f$ satisfies
\[
N(K_{a,b}, G_f) \geq \binom{b}{a} \cdot \eta_b (q-1) q^8 = \Omega_{a,b}(q^9).
\]

The proof of Lemma \ref{lem:good-plane-count} is complete.
\end{proof}

\section{Proof of Theorem~\ref{thm:main}}\label{sec:proof}

We establish the result by proving matching upper and lower bounds.

From (\ref{equation-upper_bound}), we obtain an upper bound for $\ex(n, K_{a,b}, K_{3,b+1})$. For the sake of completeness, here we attach the proof. 

Let $G$ be an $n$-vertex $K_{3,b+1}$-free graph. For every subset $T\subseteq V(G)$, let
\[
N(T)=\bigcap_{v\in T}N(v)
\]
denote the common neighborhood of $T$. Since $G$ is $K_{3,b+1}$-free, we have
\[
|N(T)|\le b
\]
for every $T\in\binom{V(G)}{3}$.

We count directed copies of $K_{a,b}$, namely ordered pairs $(A,B)$ of disjoint vertex sets with
$|A|=a$ and $|B|=b$ such that every vertex of $A$ is adjacent to every vertex of $B$. Let
$\vec{N}(K_{a,b},G)$ denote the number of such directed copies. Since each copy of $K_{a,b}$
corresponds to at most two directed copies, it suffices to estimate
$\vec{N}(K_{a,b},G)$.

Fix a directed copy $(A,B)$, and let $T\in\binom{A}{3}$. Since
$B\subseteq N(T)$ and $|N(T)|\le b$, we have
\[
N(T)=B.
\]
Thus, once $T$ is fixed, the set $B$ is uniquely determined.

Now let $T'\in\binom{B}{3}$. Observe that
\[
N(B)=\bigcap_{v\in B}N(v)\subseteq N(T').
\]
Since $|N(T')|\le b$ and $T\subseteq A\subseteq N(B)$, it follows that
\[
|N(B)\setminus T|\le b-3.
\]
Hence, after fixing $T$, the remaining $a-3$ vertices of $A$ can be chosen in at most
\[
\binom{b-3}{a-3}
\]
ways.

Finally, we count pairs $((A,B),T)$ with $T\in\binom{A}{3}$. On the one hand, each directed
copy contributes $\binom{a}{3}$ such pairs. On the other hand, there are $\binom{n}{3}$ choices
for $T$, and each determines at most $\binom{b-3}{a-3}$ directed copies. Therefore,
\[
\binom{a}{3}\vec{N}(K_{a,b},G)
\le
\binom{n}{3}\binom{b-3}{a-3},
\]
which yields
\[
\vec{N}(K_{a,b},G)
\le
\frac{\binom{n}{3}\binom{b-3}{a-3}}{\binom{a}{3}}
=
O_{a,b}(n^3).
\]
Since each copy of $K_{a,b}$ gives rise to at most two directed copies, we conclude that
\[
\ex(n,K_{a,b},K_{3,b+1})=O_{a,b}(n^3).
\]
This proves the upper bound in Theorem~\ref{thm:main}.

In the following, we prove the corresponding lower bound.
We choose an appropriate prime power to serve as the base of our finite geometric construction. Let $n$ be sufficiently large, and define
\begin{align}\label{eq8}
   y = \left(\frac{n}{2}\right)^{1/3}, \quad m = \left\lfloor \frac{y}{2} \right\rfloor. 
\end{align}
By the Bertrand-Chebyshev Theorem, 
%see \cite{chebyshev1854memoire}), 
there exists a prime integer $q$ satisfying 
\begin{align}\label{eq9}
  m<q<2m. 
\end{align}
For sufficiently large $n$, $q$ is odd and exceeds the size threshold required for Lemma~\ref{lem:double-shift}. 

We claim that the prime $q$ satisfies the two conditions for extending our construction from the  size $2q^3$ to an arbitrary large $n$-vertex graph.
On the one hand, combining (\ref{eq8}) and (\ref{eq9}), we have
$
q < 2m \leq y,
$
and thus,
$
2q^3 < 2y^3 = n,
$
which implies that the bipartite graph $G_f$ on $2q^3$ vertices has strictly smaller than $n$ vertices. We may thus obtain an $n$-vertex graph by adding $n - 2q^3$ isolated vertices to $G_f$. One can see that adding isolated vertices neither contains any copy of the forbidden subgraph $K_{3,b+1}$ nor removes any existing copy of $K_{a,b}$, so both the forbidden the subgraph count are preserved.
On the other hand, for sufficiently large $y$, combining (\ref{eq8}) and (\ref{eq9}), we have $q>m\geq y/3$. Thus by (\ref{eq8}),
we obtain
\[
q^9 > \left(\frac{y}{3}\right)^9 = \frac{y^9}{3^9} = \frac{n^3}{8 \cdot 3^9},
\]
which implies that $q^9 = \Omega(n^3)$.
By Lemma~\ref{lem:good-plane-count}, the number of $K_{a,b}$ copies in $G_f$ is $\Omega_{a,b}(q^9)$. The bound above therefore ensures that the subgraph count remains of order $\Omega_{a,b}(n^3)$ even after adding  isolated vertices.

Now applying Lemma~\ref{lem:good-plane-count} to this prime $q$, we obtain a monic degree-$b$ polynomial $f$ and the associated bipartite graph $G_f$ on $2q^3$ vertices. By Lemmas \ref{lem:k3t-free} and \ref{lem:good-plane-count}, $G_f$ is $K_{3,b+1}$-free and contains $\Omega_{a,b}(q^9)$ copies of $K_{a,b}$. We then add $n - 2q^3$ isolated vertices to $G_f$ to obtain an $n$-vertex graph $G$. Recall that adding isolated vertices neither creates any copy of $K_{3,b+1}$ nor removes any existing copy of $K_{a,b}$, $G$ is also $K_{3,b+1}$-free. So the number of $K_{a,b}$ copies in $G$ satisfies
\[
N(K_{a,b}, G) = N(K_{a,b}, G_f) = \Omega_{a,b}(q^9) = \Omega_{a,b}(n^3).
\]
This proves the lower bound
\[
\ex(n, K_{a,b}, K_{3,b+1}) = \Omega_{a,b}(n^3).
\]

Combining the upper and lower bounds, we have
\[
\ex(n, K_{a,b}, K_{3,b+1}) = \Theta_{a,b}(n^3),
\]
which completes the proof of Theorem \ref{thm:main}.

\section{Concluding remarks} \label{sec:concluding}

Our construction uses an explicit finite-field point set in $\PG(5, q)$ with a univariate 
polynomial parameter. The key new ingredient is the double-shift complete splitting 
lemma (Lemma~\ref{lem:double-shift}), which is derived from the Andrade et al's theorem on the independence of factorization types of 
shifted polynomials. This lemma allows us to show that there are many planes for which 
both the plane and its polar intersect the point set in exactly $b$ points, each such pair 
yielding a copy of $K_{b,b}$ in the incidence graph.

%It would be interesting to determine whether the threshold $t = b+1$ is best possible for all $b \geq 5$. For even $b$, this is known to be the case, but for odd $b$ the question remains open. While our result shows that taking $t=b+1$ suffices for the desired bound, it leaves open the question of whether some integer $t<b+1$ may also satisfy the required condition.

A natural direction is to extend these results to larger values of $s$. The 
general question is whether $\ex(n, K_{a,b}, K_{s,t}) = \Theta(n^s)$ whenever $t$ is 
sufficiently large compared to $b$. The cases $s = 2$ and $s = 3$ are now understood, 
but the problem for $s \geq 4$ remains open.

\end{document}